\documentclass[a4paper, 11pt]{article}
\usepackage{amsmath}
\usepackage{amssymb,esint}
\usepackage{amscd}
\usepackage{xspace}
\usepackage{fancyhdr}
\usepackage{color}
\usepackage{srcltx} 
\newtheorem{theorem}{Theorem}[section]

\newtheorem{corollary}[theorem]{Corollary}

\newtheorem{definition}[theorem]{Definition}

\newtheorem{lemma}[theorem]{Lemma}

\newtheorem{remark}[theorem]{Remark}

\newenvironment{proof}[1][Proof]{\textbf{#1.} }{\hfill\rule{0.5em}{0.5em}}
{\catcode`\@=11\global\let\AddToReset=\@addtoreset
	\AddToReset{equation}{section}
	
	\AddToReset{theorem}{section}

	\def\nc{\newcommand}
	\def\ga{\gamma}
	
	\def\lam{\lambda}\def\ep{\epsilon}\def\ka{\kappa}
	
	\def\om{\omega}

	 \def\Om{\Omega}

	\nc\pa{\partial}
	
	\nc\CC{\mathbb{C}}
	\nc\RR{\mathbb{R}}
	\nc\QQ{\mathbb{Q}}
	\nc\ZZ{\mathbb{Z}}
	\nc\NN{\mathbb{N}}

	\newcommand{\vertiii}[1]{{\left\vert\kern-0.25ex\left\vert\kern-0.25ex\left\vert #1 
			\right\vert\kern-0.25ex\right\vert\kern-0.25ex\right\vert}}

\begin{document}	
		\title{A comparison estimate for singular $p$-Laplace equations and its consequences}
		\author{
			{\bf Quoc-Hung Nguyen\thanks{E-mail address: qhnguyen@amss.ac.cn, Academy of Mathematics and Systems Science, Chinese Academy of Sciences, Beijing, 100190, China.} ~ and~ Nguyen Cong Phuc\thanks{E-mail address: pcnguyen@math.lsu.edu, Department of Mathematics, Louisiana State University, 303 Lockett Hall,
					Baton Rouge, LA 70803, USA. }}}
		\date{}  
		\maketitle
		\begin{abstract}
		Comparison estimates are an important technical device in the study  of regularity problems for quasilinear possibly degenerate elliptic and parabolic equations. Such tools have been employed 
		indispensably in many papers of Mingione, Duzaar-Mingione, and Kuusi-Mingione, etc. on certain measure datum problems to obtain pointwise bounds for solutions and their full or fractional derivatives in terms of appropriate linear or nonlinear potentials. However, a comparison estimate for  $p$-Laplace type elliptic equations with measure data
		is still unavailable in the strongly singular case $1< p\leq \frac{3n-2}{2n-1}$, where $n\geq 2$ is the dimension of the ambient space. This issue will be completely resolved
		in this work by proving a comparison estimate in a slightly larger range $1<p<3/2$.   Applications include 	a `sublinear'  Poincar\'e type inequality, pointwise bounds for solutions and their derivatives by Wolff's and Riesz's potentials, respectively.  	Some global pointwise and  weighted estimates are also obtained for bounded domains, which enable us to treat a  quasilinear Riccati type equation with possibly sublinear growth in the gradient. 
				
			\medskip
			
			\medskip
			
			\medskip
			
			
			\medskip
			
		\end{abstract}   
		
		\tableofcontents
		
		\section{Introduction and main results} 
		
		In this paper, we are concerned with  the quasilinear elliptic equation with measure data 
		\begin{equation}\label{quasi-measure}
		-\operatorname{div}(A(x,\nabla u))=\mu
		\end{equation}
	
		in a bounded open subset $\Omega$ of $\mathbb{R}^n$, $n\geq 2$. Here $\mu$ is a finite signed measure in $\Omega$ and
		the nonlinearity  $A = (A_1, \dots, A_n):\mathbb{R}^n\times \mathbb{R}^n\to \mathbb{R}^n$ is a vector valued function.\\
		Our main assumptions on $A$ are as follows. There exist $\Lambda\geq 1$ and $p\in (1,n)$ such that
		for every $x\in \mathbb{R}^n$ and every $(\xi,\eta)\in \mathbb{R}^n\times \mathbb{R}^n\backslash\{(0,0)\}$,
		\begin{equation}
		\label{condi1}| A(x,\xi)|\le \Lambda |\xi|^{p-1}, \quad | D_\xi A(x,\xi)|\le \Lambda |\xi|^{p-2},
		\end{equation}
		\begin{equation}
		\label{condi2}  \langle D_\xi A(x,\xi)\eta,\eta\rangle\geq \Lambda^{-1}  |\xi|^{p-2} |\eta|^2.
		\end{equation}
		A typical model for \eqref{quasi-measure} is obviously given by the $p$-Laplace equation with measure data
		\begin{equation*}
		-\Delta_p \, u:= -{\rm div}(|\nabla u|^{p-2} \nabla u)=\mu \quad \text{in~} \Omega.
		\end{equation*}
		For a reason that will be explained momentarily, unless otherwise stated,  in this paper we only consider the `singular' case 
		\begin{equation*}
		1<p<3/2.
		\end{equation*}
		An important technical tool in the study of  regularity problems for equation \eqref{quasi-measure} is a comparison
		estimate that connects the solution 
		of measure datum problem to a solution of a homogeneous problem. 
		To describe it,  let $Q_r(x_0)$  denote the open cube $Q_r(x_0):= x_0+ (-r,r)^n$ with center $x_0\in\RR^n$ and side-length  $2r$.  
		Let  $u\in W_{\rm loc}^{1,p}(\Omega)$ be a solution of \eqref{quasi-measure}.
		For  a cube $Q_r=Q_{r}(x_0) \Subset\Omega$, we  then consider the  unique solution $w\in W_0^{1,p}(Q_r(x_0))+u$ to the local interior problem

		\begin{equation}\label{eq1}
		\left\{ \begin{array}{rcl}
		- \operatorname{div}\left( {A(x,\nabla w)} \right) &=& 0 \quad \text{in} \quad Q_r(x_0), \\ 
		w &=& u\quad \text{on} \quad \partial Q_r(x_0).  
		\end{array} \right.
		\end{equation}

		For $p>\frac{3n-2}{2n-1}$, this comparison estimate reads as follows.

		\begin{lemma}[\cite{Duzamin2, M07, HPanna}]\label{com1} Assume that $p>\frac{3n-2}{2n-1}$. Then we have  
			\begin{align}\label{oldcom}
			\left(	\fint_{Q_{r}}|\nabla u-\nabla w|^{\gamma}dx\right)^{\frac{1}{\gamma}} \leq C \left[\frac{|\mu|(Q_{r})}{r^{n-1}}\right]^{\frac{1}{p-1}} +C\frac{|\mu|(Q_{r})}{r^{n-1}}\left(\fint_{Q_{r}}|\nabla u|^{\gamma}\right)^{\frac{2-p}{\gamma}}dx.
			\end{align}	
		\end{lemma}	
		Here $\gamma=1$ in the case $p>2-1/n$ and when $p\geq 2$ the second term on the right can be dropped  (see \cite{Duzamin2}, \cite{M07}). 
		For $\frac{3n-2}{2n-1}<p\leq 2-\frac{1}{n}$, $ \gamma\in (\frac{n}{2n-1}, \frac{n(p-1)}{n-1})$ (see \cite{HPanna}).

		In fact, as demonstrated in our recent work \cite[Lemma 2.1]{HP2}, the comparison bound \eqref{oldcom} also holds for all $1<p\leq \frac{3n-2}{2n-1}$, as long as we  replace 
		the quantity $(\fint_{Q_{r}}|\nabla u|^{\gamma}dx)^{(2-p)/\gamma}$ on the right-hand side with the  quantity
		$\fint_{Q_{r}}|\nabla u|^{2-p}dx$. The power $|\nabla u|^{2-p}$ arrives naturally from the structure of $p$-Laplace type equations when $p<2$, (see \eqref{gp}). However, the quantity
		$\fint_{Q_{r}}|\nabla u|^{2-p}dx$ is useful only when $2-p< \frac{n(p-1)}{n-1}$ as $\frac{n(p-1)}{n-1}$ is the optimal (weak) integrability of the gradient of  fundamental solution. This shows that  the quantity
		$\fint_{Q_{r}}|\nabla u|^{2-p}dx$ is `super-critical' in the singular case $1<p\leq \frac{3n-2}{2n-1}$.

		The  main result of this paper is the extension of the comparison estimate to the remaining open case $1<p\leq \frac{3n-2}{2n-1}$. As $\frac{3n-2}{2n-1}<3/2$, it is enough to consider the range $1<p<3/2$ mentioned above. In what follows, we shall fix a constant $\kappa$ defined by
		$$\kappa=(p-1)^2/2.$$

		\begin{theorem}\label{maininterior} Suppose that $n\geq 2$ and $Q_{\Sigma \,r}(x_0)\subset\Om$ for $r>0$ and $\Sigma \in (1,2]$. Let  $u\in W_{\rm loc}^{1,p}(\Omega)$ be a solution of \eqref{quasi-measure} and let $w$ be as \eqref{eq1}. Then 
			\begin{align}\label{compgradu}
			&\left(	\fint_{Q_r(x_0)}|\nabla (u-w)|^{\ka}\right)^{\frac{1}{\kappa}} + \frac{1}{r}\left(	\fint_{Q_r(x_0)}|u-w|^{\ka}\right)^{\frac{1}{\kappa}}\nonumber\\
			 &\qquad \lesssim \left(\frac{|\mu|(Q_{\Sigma r}(x_0))}{r^{n-1}}\right)^{\frac{1}{p-1}}   + \frac{|\mu|(Q_{\Sigma r}(x_0))}{r^{n-1}}	
			\left(	\fint_{Q_{\Sigma r}(x_0)}|\nabla u|^{\ka}\right)^{\frac{2-p}{\kappa}},
			\end{align} 
			and for any $\lambda \in \RR$,
			\begin{align*}
			\left(	\fint_{Q_r(x_0)}|u-w|^{\ka}\right)^{\frac{1}{\kappa}} & \lesssim\left(\frac{|\mu|(Q_{\Sigma r}(x_0))}{r^{n-p}}\right)^{\frac{1}{p-1}}+ \frac{|\mu|(Q_{\Sigma r}(x_0))}{r^{n-p}}	
			\left(	\fint_{Q_{\Sigma r}(x_0)}|u-\lambda|^{\ka}\right)^{\frac{2-p}{\kappa}}.
			\end{align*}
		\end{theorem}
		
						A boundary version of Theorem  \ref{maininterior}, where $x_0\in \partial\Om$ and $u\in W^{1,p}_0(\Om)$, is also available; see Theorem 
			\ref{mainboundary} below. Theorem \ref{maininterior} follows from  Lemmas \ref{comsqrt3},  \ref{revholderformixgrad}, and Corollary \ref{co234} below. Notice that we 
		 use the larger cube $Q_{\Sigma r}(x_0)$ on the right-hand side of the bounds in Theorem \ref{maininterior}, which makes it slightly different from Lemma \ref{com1}.
			But this is harmless in applications.  As $\kappa<\frac{n(p-1)}{n-1}$ when $1<p<2$, we see that the quantity $(\fint_{Q_{\Sigma r}(x_0)}|\nabla u|^{\ka})^{1/\kappa}$ is sub-critical
			in this case. In fact, as the proof goes, in the first step (see Lemma \ref{comsqrt3}) we obtain an inequality similar to \eqref{compgradu} but with a mixed norm quantity 
			\begin{equation}\label{mixednormquan}
			\left(\fint_{x_0'+(-r,r)^{n-1}}\left(\fint_{x_{0n}+(-r,r)} |\nabla u (x',x_n)|^{2-p} dx_n\right)^{\frac{p-1}{3-p}} dx'\right)^{\frac{3-p}{(p-1)(2-p)}}
			\end{equation}
			in place of $(\fint_{Q_{\Sigma r}(x_0)}|\nabla u|^{\ka})^{1/\kappa}$, and no enlargement of the cube is needed in this step. 
			Note that this is a substantial improvement of \cite[Lemma 2.1]{HP2} as the exponent $2-p$ is used only in one direction and the exponent for the other directions goes to zero as  $p$ approaches  $1$.  
			Another way to look at this is to consider the fundamental solution $v(x)=c\, |x|^{(p-n)/(p-1)}$, $1<p<n$, of the $p$-Laplace equation. Then $|\nabla v|\in L^{2-p}_{\rm loc}$
			if and only if $p>\frac{3n-2}{2n-1}$, whereas the mixed norm above for $\nabla v$ is finite for all $p>1$. We indeed show in Lemma  \ref{revholderformixgrad} (see also Corollary \ref{co234})  that the latter also holds 
			for the gradient of  any solution to \eqref{quasi-measure} by using the first step and a suitable reverse H\"older type inequality for gradients of solutions of the homogeneous equation
			 ${\rm div}(A(x, \nabla w))=0$. Moreover, this enables us to control the quantity in \eqref{mixednormquan} by $(\fint_{Q_{\Sigma r}(x_0)}|\nabla u|^{\ka})^{1/\kappa}$ 
			 in some way to complete the proof.
		
		For the rest of this section, we will present some selected consequences of Theorem \ref{maininterior}. More consequences  of Theorem \ref{maininterior} will be presented elsewhere in our future work.
		
		First, we have a Poincar\'e type inequality with low integrability and with a `correction' term.  It is well-known that Poincar\'e  inequality generally fails in the `sublinear' range (see \cite{BK}).  
		
		\begin{corollary}\label{poicaresub} Suppose that $n\geq 2$ and $Q_{\Sigma \,r}(x_0)\subset\Om$ for $r>0$ and $\Sigma \in (1,2]$. 
			Let  $u\in W^{1,p}_{\rm loc}(\Om)$ be a solution of \eqref{quasi-measure}. Then for any $\epsilon>0$ we have   
			\begin{align*}
			\left(	\inf_{q\in \RR}\fint_{Q_r(x_0)}|u-q|^{\ep}\right)^{\frac{1}{\ep}} \lesssim\left(\frac{|\mu|(Q_{\Sigma r}(x_0))}{r^{n-p}}\right)^{\frac{1}{p-1}}+ 	
			r \left(	\fint_{Q_{\Sigma r}(x_0)}|\nabla u|^{\ep}\right)^{\frac{1}{\ep}}.
			\end{align*}
		\end{corollary}

		Second, we have a result on potential estimates for solutions of  \eqref{quasi-measure}. In particular, Theorem \ref{upw} below 
		 extends \cite[Theorems 1.1 and 1.10]{KM1} to all smaller values of $p$. For nonnegative measures $\mu$, \eqref{Wpointwiseu}  is a fundamental result due to Kilpel\"ainen and Mal\'y \cite{KM} (see also \cite{TW}). 
		  Hereafter, the (truncated) Havin-Maz'ya-Wolff's potential (often called Wolff's potential) ${\bf W}_{\gamma,p}^{R}(\nu)$, $R, \gamma>0$, of a nonnegative measure $\nu$ is defined by
		 $${\bf W}_{\gamma,p}^{R}(\nu)(x):= \int_{0}^R \left[\frac{\nu(B_t(x))}{t^{n-\gamma p}}\right]^{\frac{1}{p-1}}\frac{dt}{t}, \quad x\in\RR^n.$$
		On the other hand, the truncated Riesz's potential of order $\gamma>0$ is defined as
		 $${\bf I}^{R}_\gamma(\nu) := {\bf W}_{\gamma/2,2}^{R}(\nu)(x)= \int_{0}^R \frac{\nu(B_t(x))}{t^{n-\gamma }}\frac{dt}{t}, \quad x\in\RR^n.$$
		\begin{theorem}\label{upw}
			Suppose that $u\in C^0(\Omega)\cap W^{1,p}_{\rm loc}(\Om)$ solves \eqref{quasi-measure} for a finite measure $\mu$ in $\Omega$. Then under \eqref{condi1}--\eqref{condi2}, 
			for any $B_{R}(x_0)\subset \Om$	we have 
			\begin{equation}\label{Wpointwiseu}
			|u(x_0)| \lesssim{\bf W}_{1,p}^{R}(|\mu|)(x_0) + \Big(\fint_{B_{R}(x_0)}|u(z)|^{\kappa} dz\Big)^{\frac{1}{\kappa}}.
			\end{equation}
		\end{theorem}
		

		Third, we discuss potential estimates for gradients of  solutions to \eqref{quasi-measure}. Pointwise estimates for the  gradient 
		were  obtained in \cite{55DuzaMing,Duzamin2,KM2,KM3} for the case $p>2-1/n$.  For the case $\frac{3n-2}{2n-1}<p\leq 2-1/n$, see the recent work \cite{DZ} (see also \cite{HP}). The next theorem covers all of the remaining range of $p$.
		\begin{theorem}\label{gradpw} Assume that $A({x, \xi})$ satisfies the following Dini  condition:
			\begin{equation}\label{integralmodul}
			\int_0^1 \om(\rho) \frac{d\rho}{\rho} <+\infty
			\end{equation}
			for a non-decreasing function $\om:[0,1]\rightarrow [0, \infty)$ such that $\lim_{\rho\rightarrow 0^{+}} \om(\rho)=\om(0)=0$ and
			$$	|A(x,\xi)-A(y,\xi)| \leq \om(|x-y|) |\xi|^{p-1}, \qquad \forall x,y,\xi\in \RR^n, |x-y|\leq 1.$$
			Suppose that $u\in C^1(\Omega)$ solves \eqref{quasi-measure} for a finite measure $\mu$ in $\Omega$. Then 
			for any	ball $B_R(x_0)\subset\Omega$ we have 
			\begin{equation*}
			|\nabla u(x_0)| \lesssim \Big[ {\bf I}_{1}^{R}(|\mu|)(x_0)\Big]^{\frac{1}{p-1}} + \Big(\fint_{B_{R}(x_0)}|\nabla u(y)|^{\kappa} dy\Big)^{\frac{1}{\kappa}}.
			\end{equation*}
		\end{theorem}

		The proof of Theorem \ref{gradpw} is based mainly on Theorem \ref{maininterior} and the pioneering idea of \cite{Duzamin2} (see also \cite{KM1}). However, some modifications are needed due to the fact that the gradient of solution may not belong to $L^1_{{\rm loc}}(\Om)$.  Such modifications have been carried out in \cite{DZ} for $\frac{3n-2}{2n-1}<p\leq 2-1/n$ (see also \cite{HP}). 	The main point here  is to replace the mean oscillations such as 
		$$\fint_{B_\rho} |\nabla u- (\nabla u)_{B_\rho}| dy, \qquad (\nabla u)_{B_\rho}:=\fint_{B_\rho} \nabla u dz,$$ 
		with the quantity 
		$$\left(\inf_{\mathbf{q}\in\RR^n} \fint_{B_\rho}|\nabla u- \mathbf{q}|^\kappa dy\right)^{\frac{1}{\kappa}}.$$
		
		Thus we shall omit the details of the proof of Theorem \ref{gradpw}. 
%

		Fourth, we discuss some global gradient estimates for {\it renormalized solutions} to the Dirichlet problem
		\begin{eqnarray}\label{Di}
		\left\{ \begin{array}{rcl}
		-{\rm div}(A(x, \nabla u))&=&  \mu \quad \text{in} ~\Omega, \\
		u&=&0  \quad \text{on}~ \partial \Omega.
		\end{array}\right.
		\end{eqnarray}
		
		We refer to the paper \cite{11DMOP} for  several equivalent definitions of renormalized solutions along with their stability property. Here we mention that truncations of a renormalized solution $u$ of \eqref{Di} are stable even near the boundary in the sense that if $u_k=T_k(u)$, $k>0$, where 
		\begin{equation}\label{Trun}
		T_k(s)=\max\{\min\{s,k\},-k\}, \qquad s\in\mathbb{R},
		\end{equation}
		then $u_k\in W^{1,p}_0(\Om)$ is the unique solution of \eqref{Di} with  measure datum $\mu_k$ in place of $\mu$ such that 
		$\mu_k$ converges in the narrow topology of measures to $\mu$, i.e., 
		$$\lim_{k\rightarrow\infty}\int_{\Omega}\varphi \, d\mu_{k}=\int_{\Omega}\varphi \,
		d\mu,$$
		for every {\it bounded and continuous} function $\varphi$ on $\Omega$.

		Pointwise gradient estimates  up to the boundary for renormalized solutions of  \eqref{Di} can be deduced from \cite{DZ} in the case $p\in (\frac{3n-2}{2n-1}, 2)$.
		Here  $A(x,\xi)$ is assumed to have the special structure $A(x,\xi)=a(x)|\xi|^{p-2}\xi$ for any $x$ near  $\partial\Om$, and the boundary 
		$\partial\Om$ is assumed to be of class $C^{1, \, {\rm Dini}}$ (see \cite{DZ}). 	 With Theorem \ref{maininterior} and its boundary counterpart (Theorem \ref{mainboundary}) at hand,  the method of    \cite{DZ} also applies 
		to the case $1<p<3/2$. For that reason, we shall not present the details of the proof.
		
		\begin{theorem}\label{globalpw} Assume that $A({x, \xi})$ satisfies the  Dini  condition \eqref{integralmodul}, and that $\partial \Om$ is of class $C^{1,\alpha}$ 
		or even $C^{1, \, {\rm Dini}}$. 	
		Assume also that 	$A(x,\xi)=a(x)|\xi|^{p-2}\xi$ for any $x$ near  $\partial\Om$. Then for any renormalized solution $u$ of  \eqref{Di}
		we have  
			\begin{equation}\label{Wpointwiseu"'}
			|\nabla u(x)| \leq C_0\Big[ {\bf I}_{1}^{2 {\rm diam}(\Om)}(|\mu|)(x)\Big]^{\frac{1}{p-1}} \qquad {\rm a.e.~} x\in\Om. 
			\end{equation}
		\end{theorem}

		On the other hand, in many applications it is enough 	to use a weighted integral estimate of Muckenhoupt-Wheeden type  for the gradient, which we shall describe next. For such weighted integral bounds, we can
		work with equation \eqref{Di}  under a much weaker
		condition on the coefficients and the domain. 
		
		As for the coefficients, we recall the following definition from \cite{Ph2}.
			\begin{definition} Given  $\delta, R_0>0$, we say that $A({x, \xi})$ satisfies the $(\delta, R_0)$-BMO condition if
			\begin{equation*}
			\sup_{y\in\RR^n, \, 0<r\leq R_0 }  \fint_{B_{r}(y)}\Upsilon(A,
			B_{r}(y))(x) dx  \leq \delta,
			\end{equation*}
			where  
$$\Upsilon(A,B_r(y))(x) := \sup_{\xi \in \mathbb{R}^{n}\setminus \{0\}} \frac{|A({x, \xi}) - \fint_{B_r(y)}A(z, \xi)dz|}{|\xi|^{p-1}}.$$
		\end{definition}

		As for the domain, we shall use the notion of Reifenberg flat domain (see \cite{55Re}). This includes  $C^1$ and  Lipschitz domains (with sufficiently small Lipschitz constants), as well as certain fractal domains.
				 
		\begin{definition}
			Given $\delta\in (0, 1)$ and $R_0>0$, we say that $\Omega$ is a $(\delta, R_0)$-Reifenberg flat domain if for every $x\in \partial \Omega$
			and every $r\in (0, R_0]$, there exists a
			system of coordinates $\{ z_{1}, z_{2}, \dots,z_{n}\}$,
			which may depend on $r$ and $x$, so that  in this coordinate system $x=0$ and that
			\[
			B_{r}(0)\cap \{z_{n}> \delta r \} \subset B_{r}(0)\cap \Omega \subset B_{r}(0)\cap \{z_{n} > -\delta r\}.
			\]
		\end{definition}
		
		Our weighted integral estimate is obtained for the class of $\mathbf{A}_\infty$ weights. This class consists of nonnegative  functions $w\in L^1_{\text{loc}}(\mathbb{R}^{n})$ such that there are two positive constants $C$ and $c$ such that
		$$w(E)\le C \left(\frac{|E|}{|B|}\right)^{c} w(B),
		$$
		for all balls $B$ and all measurable subsets $E$ of $B$. The pair $(C,c) $ is called the $\mathbf{A}_\infty$ constants of $w$ and is denoted by $[w]_{\mathbf{A}_\infty}$.

The following theorem extends the results of \cite{Ph2, HPanna} (see also \cite{HP2}) to the remaining   case $1<p\leq\frac{3n-2}{2n-1}$. 

		\begin{theorem} \label{101120143}   For any $w\in \mathbf{A}_{\infty}$ and $0< q<\infty$, there exists    $\delta=\delta(n,p,\Lambda, q, [w]_{\mathbf{A}_{\infty}})\in (0,1)$ such that if $A(x,\xi)$ is $(\delta, R_0)$-BMO and $\Omega$ is  $(\delta,R_0)$-Reifenberg flat  for some $R_0>0$, then  for any renormalized solution $u$ of \eqref{Di}, we have                            
			\begin{equation}\label{wLq}
			\|\nabla u\|_{L^{q}_w(\Omega)}\leq C \|[\mathbf{M}_1(\mu)]^{\frac{1}{p-1}}\|_{L^{q}_w(\Omega)}.
			\end{equation} 
			Here the constant $C=C(n,p,\Lambda,q, [w]_{\mathbf{A}_\infty}, diam(\Omega)/R_0)$, and        $\mathbf{M}_1(\mu)$   is a fractional maximal function $\mu$ defined by
			$$\mathbf{M}_1(\mu)(x):= \sup_{r>0} \frac{|\mu|(B_t(x))}{t^{n-1}}, \qquad x\in\RR^n.$$     
		\end{theorem}
	
	We remark that the weighted bound \eqref{wLq} also holds if we replace the weighted space $L^q_w(\Om)$ with the weighted Lorentz space $L^{q,s}_w(\Om)$ for any 
	$0<s\leq \infty$.
		
	A proof of 	Theorem \ref{101120143} in the case $p>2-1/n$ and the case $\frac{3n-2}{2n-1}<p\leq 2-1/n$ was presented in \cite{Ph2} and \cite{HPanna}, respectively.
	Now with the availability of the interior and boundary comparison estimates (Theorems \ref{maininterior} and \ref{mainboundary}), the proof also works in the present case $1<p<3/2$.  For that, we shall not repeat the proof here.
	We mention that, following the approach of \cite{HPanna}, a `good-$\lambda$' type inequality  involving  
	$\mathbf{M}_1(\mu)^{1/(p-1)}$ and ${\bf M}(|\nabla u|^\kappa)^{1/\kappa}$, where ${\bf M}$ is the Hardy-Littlewood maximal function,	can also be obtained as in \cite[Theorem 1.5]{HPanna}.

	Finally,  as an application of Theorem \ref{globalpw} we obtain a sharp existence result for a quasilinear Riccati type equation with measure data:																																					
		\begin{eqnarray}\label{Riccati}
		\left\{ \begin{array}{rcl}
		-{\rm div}(A(x, \nabla u))&=& |\nabla u|^q + \mu \quad \text{in} ~\Omega, \\
		u&=&0  \quad ~~~~~~~~~~~\text{on}~ \partial \Omega.
		\end{array}\right.
		\end{eqnarray}
		To this end, we shall need the notion of capacity
		associated to the Sobolev space $W^{1, s}(\RR^n)$,  $1<s<+\infty$, defined for each
		 compact set $K\subset\RR^n$ by
		\begin{equation*}
		{\rm Cap}_{1,  s}(K)=\inf\Big\{\int_{\RR^n}(|\nabla \varphi|^s +\varphi^s) dx: 0\leq \varphi\in C^\infty_0(\RR^n),
		\varphi\geq 1 {\rm ~on~} K \Big\}.
		\end{equation*}

		\begin{theorem}\label{main-Ric}
			Let $q> p-1$. Assume that \eqref{integralmodul} holds and $\partial\Om$ is of class $C^{1, \, {\rm Dini}}$. Also, assume that  $A(x,\xi)=a(x)|\xi|^{p-2}\xi$ for any $x$ near the boundary $\partial\Om$. 
			Then there exists a constant 
			$c_0>0$ such that if the measure $\mu$ satisfies 
			\begin{equation}\label{capcondi} 
			|\mu|(K) \leq c_0\, {\rm Cap}_{1,\, \frac{q}{q-p+1}}(K)
			\end{equation}
			for all compact sets $K\subset\Om$, then there exists a renormalized solution $u$ to the Riccati type equation \eqref{Riccati} such that 
			\begin{equation*}
			|\nabla u(x)| \lesssim \Big[ {\bf I}_{1}^{2{\rm diam}(\Om)}(|\mu|)(x)\Big]^{\frac{1}{p-1}}\qquad {\rm a.e.~} x\in\Om.
			\end{equation*}
		\end{theorem}
	
We remark that condition \eqref{capcondi} is sharp at least for nonnegative measures with compact support in $\Om$ (see \cite{HMV, Ph1}). It is worth mentioning that 
Theorem \ref{main-Ric} completely solve a problem raised in \cite[pages 13--14]{BGV}. 
 On the other hand,  it is new mainly in the   case $\frac{n(p-1)}{n-1}\leq q<1$ as the case $p-1<q<\frac{n(p-1)}{n-1}$ and the case $q\geq 1$ have been treated in \cite{GMP} and \cite{HP2}, respectively. For other earlier work on  equation \eqref{Riccati}, we refer to the papers \cite{HMV, Ph1, Ph2, Ph3, MP, HP3} and the references therein. 
		
		\section{Comparison estimates}

	In what follows,  for  $x\in\RR^n$ we write $x=(x',x_n) \in \RR^{n-1}\times \RR$. With this,  
	we will use the notation $\|f\|_{L^{s_2}_{x'}L^{s_1}_{x_n}(Q_\rho(x_0))}$ for $s_1, s_2>0$ and $\rho>0$, to indicate the mixed norm
	$$\|f\|_{L^{s_2}_{x'}L^{s_1}_{x_n}(Q_\rho(x_0))}=\left(\int_{x_0'+(-\rho,\rho)^{n-1}}\left(\int_{x_{0n}+(-\rho,\rho)} |f(x',x_n)|^{s_1} dx_n\right)^{\frac{s_2}{s_1}} dx'\right)^{\frac{1}{s_2}}. $$
On the other hand, we write $\vertiii{f}_{L^{s_2}_{x'}L^{s_1}_{x_n}(Q_\rho(x_0))}$ for the normalized mixed norm
	$$\vertiii{f}_{L^{s_2}_{x'}L^{s_1}_{x_n}(Q_\rho(x_0))}=\left(\fint_{x_0'+(-\rho,\rho)^{n-1}}\left(\fint_{x_{0n}+(-\rho,\rho)} |f(x',x_n)|^{s_1} dx_n\right)^{\frac{s_2}{s_1}} dx'\right)^{\frac{1}{s_2}}. $$
	We begin with the following lemma.
	
	\begin{lemma}\label{test} Let $u\in W^{1,p}_{\rm loc}(\Om)$  be a solution of \eqref{quasi-measure}. Then for any  function $\phi\in C_0^\infty(\Om)$, $\phi\geq0$, and $\epsilon>0$, we have  	
		\begin{align*}
		 \int_{\Om} \left|\nabla \left[(1+|u|)^{\frac{p-1-\epsilon}{p}} \phi\right]\right|^p dx & \lesssim \int_{\Om} \phi^p d|\mu| + \int_{\Omega} (1+|u|)^{(p-1)(1+\epsilon)}  |\nabla \phi|^{p} dx. 
		\end{align*}	
	\end{lemma}
	\begin{proof} Let $\phi$ be a function in $C_0^\infty(\Om)$ such that $\phi\geq 0$.
		Then using $${\rm sign}(u)[1-(1+|u|)^{-\epsilon}] \phi^p, \qquad \epsilon>0,$$ as a test function in \eqref{quasi-measure} we have 	
		\begin{align*}
		&\epsilon \int_{\Om} \frac{|\nabla u|^p}{(1+ |u|)^{1+\epsilon}} \phi^p dx  \lesssim \int \phi^p d|\mu| + p\int_{\Omega}|\nabla u|^{p-1} \phi^{p-1} |\nabla \phi|  dx \\
		&= \int_\Om \phi^p d|\mu| + p\int_{\Om}|\nabla u|^{p-1} \phi^{p-1}  (1+|u|)^{(-\epsilon-1)\frac{p-1}{p}}   (1+|u|)^{(\epsilon+1)\frac{p-1}{p}} |\nabla \phi|  dx.
		\end{align*}
		Thus by Young's inequality we obtain 
$$
		\int_{\Om} \frac{|\nabla u|^p}{(1+ |u|)^{1+\epsilon}} \phi^p dx  \lesssim \int_{\Om} \phi^p d|\mu| + \int_{\Omega} (1+|u|)^{(p-1)(1+\epsilon)}  |\nabla \phi|^{p} dx. $$
	 Now in view of the identity
		$$\nabla \left[(1+|u|)^{\frac{p-1-\epsilon}{p}} \phi\right]= \tfrac{p-1-\epsilon}{p} (1+|u|)^{\frac{-1-\epsilon}{p}} \nabla u\,  {\rm sign}(u) \phi + (1+|u|)^{\frac{p-1-\epsilon}{p}} \nabla \phi,$$
		we obtain the lemma.
			\end{proof}
	
	Lemma \ref{test} yields the following reverse H\"older type inequality for the solution.
	
	\begin{lemma}\label{ureverse} Let  $u\in W^{1,p}_{\rm loc}(\Om)$ be a solution of \eqref{quasi-measure}. Then for any   cube $Q_R(x_0)\subset\Om$ and any $q>q_1>0$ we 
		we have  	
		\begin{align*}
		\left(\fint_{Q_{\sigma R}(x_0)}|u|^q dx\right)^{\frac{1}{q}} \lesssim \left(\frac{|\mu|(Q_R(x_0))}{R^{n-p}}\right)^{\frac{1}{p-1}} + \left(\fint_{Q_{R}(x_0)} |u|^{q_1} dx\right)^{\frac{1}{q_1}},
		\end{align*}
		provided $q<\frac{n(p-1)}{n-p}$ and $\sigma\in (0,1)$.
	\end{lemma}
	\begin{proof} 
		By translation and scale invariance, we may assume that $x_0=(0,\dots,0)$, $Q_R(x_0)=(-1,1)^n$, and 
		\begin{equation}\label{scalecondfirst}
		|\mu|((-1,1)^n)^{\frac{1}{p-1}} + \|u\|_{L^{q_1}((-1,1)^n)} \lesssim 1.
		\end{equation}	
		Moreover, we  need to show that for any $\sigma<1$,
		$$\|u\|_{L^q((-\sigma,\sigma)^n)} \lesssim 1; $$
		see \cite[Remark 4.1]{Duzamin2} for details.\\Let $\epsilon>0$ and let $\phi$ be a function in $C_0^\infty(Q_{r}(0))$ such that $0\leq \phi\leq 1$, $\phi\equiv 1$ on $Q_{s}(0)$, $1\geq r>s>0$, and $|\nabla \phi|\leq C/(r-s).$ By Lemma \ref{test} and Sobolev embedding theorem we have 
$$
		\|(1+|u|)^{\frac{p-1-\epsilon}{p}} \phi\|^p_{L^{p^{*}}(\Om)} 
		\lesssim \int_\Om \phi^p d|\mu|  + \int_{\Omega} (1+|u|)^{(p-1)(1+\epsilon)}  |\nabla \phi|^{p} dx,$$
		where $p^*=\frac{np}{n-p}$. 	This gives, in view of \eqref{scalecondfirst}, 
$$
		\left(\int_{Q_s(0)}(1+|u|)^{\frac{(p-1-\epsilon)p^{*}}{p}}dx\right)^{\frac{p}{p^*}}  \lesssim 1  + \frac{1}{(r-s)^p}\int_{Q_r(0)} (1+|u|)^{(p-1)(1+\epsilon)} dx.$$
		Let $$q_\epsilon:=\frac{(p-1-\epsilon)p^{*}}{p}, \qquad \tilde{q}_\ep:= (p-1)(1+\epsilon),$$
		and choose $\ep$ sufficiently small so that 
		 $q_\ep\in\left [q,\frac{n(p-1)}{n-p}\right )$, and $\tilde{q}_\ep <q_\epsilon.$\\
		 Then we can rewrite the above inequality as 
$$\int_{Q_s(0)}(1+|u|)^{q_\epsilon }dx \lesssim 1  + \frac{1}{(r-s)^{p^*}} \left(\int_{Q_r(0)} (1+|u|)^{\tilde{q}_\ep} dx\right)^{p^*/p}.$$
		Also, we may assume that $q_1< p-1 < \tilde{q}_\epsilon$ and then by H\"older's inequality  and \eqref{scalecondfirst} we have 
		\begin{align*}
		\left(\int_{Q_r(0)} (1+|u|)^{\tilde{q}_\ep} dx\right)^{\frac{1}{\tilde{q}_\ep}} &\leq \left(\int_{Q_r(0)} (1+|u|)^{q_\ep} dx\right)^{\frac{\theta}{q_\ep}} \left(\int_{Q_r(0)} (1+|u|)^{q_1} dx\right)^{\frac{1-\theta}{q_1}}\\
		&\leq \left(\int_{Q_r(0)} (1+|u|)^{q_\ep} dx\right)^{\frac{\theta}{q_\ep}},
		\end{align*}
		where 
		$\theta=\frac{(\tilde{q}_\ep-q_1) q_\ep}{(q_\ep-q_1) \tilde{q}_\ep}.$\\
		We next further restrict $\epsilon<\frac{q_1}{n}$ so that 
		$$\frac{\theta  \tilde{q}_\ep}{q_\ep} \frac{ p^{*}}{p}= \frac{\tilde{q}_\ep-q_1}{q_\ep-q_1} \frac{n}{n-p}<1.$$ 
		Then by Young's inequality with exponents
		$\frac{pq_\ep}{\theta \tilde{q}_\ep p^{*}},$ and  $\frac{pq_\ep}{pq_\ep- \theta \tilde{q}_\ep p^{*}},$
		we get 
$$
		\int_{Q_s(0)}(1+|u|)^{q_\ep}dx \leq C 	   +  \frac{1}{2}\int_{Q_r(0)}(1+|u|)^{q_\ep}dx
		+ \frac{C}{(r-s)^{\mathcal{H}_0}},$$
		where 
		$$ \mathcal{H}_0=\frac{p^{*} pq_\ep}{pq_\ep-\theta\tilde{q}_\ep p^{*}}=\frac{np(q_\ep-q_1)}{n(q_\ep-\tilde{q}_\ep)-p(q_\ep-q_1)}>0.$$
		Now we may apply Lemma 6.1 of \cite{Giu} to obtain 
$$
		\int_{Q_\sigma(0)}(1+|u|)^{q_\ep}dx \leq C,$$
		which of course yields the desired result. 	
	\end{proof}\\
	We next give a mixed norm estimate for the gradient.
	\begin{lemma}
		Let  $Q_r(x_0)\subset\Om$ and $1-\frac{n-1}{p}<\alpha<1/p$, $\alpha>0$, $n\geq 2$.  Then for any $\sigma\in(0,1)$ and $q_1>0$ we have 
$$\vertiii{\nabla u}_{L^{p_2}_{x'}L_{x_n}^{p_1}(Q_{\sigma r}(x_0))}\lesssim \left(\frac{|\mu|(Q_r(x_0))}{r^{n-1}}\right)^{\frac{1}{p-1}} + \frac{1}{r} \left(\fint_{Q_r(x_0)}|u|^{q_1} dx\right)^{\frac{1}{q_1}},$$
		where  $\gamma\in(1,\min\{p,n/(n-p)\})$, and 
		 $$p_1=\frac{p-\gamma}{1-\alpha \gamma}, \qquad p_2=\frac{(n-1)(p-\gamma)}{n-1-(1-\alpha)\gamma}.  $$	
	\end{lemma}
	\begin{proof} By translating and scaling considerations, we may assume that $x_0=(0,\dots,0)$, $Q_r(x_0)=Q_1(0)=(-1,1)^n$, and \eqref{scalecondfirst}
	holds.	
	Moreover, we  need to show that for any $\sigma<1$
	$$\|\nabla u\|_{L^{p_2}_{x'}L_{x_n}^{p_1}(Q_{\sigma}(0))} \lesssim 1. $$
		By Lemma \ref{test}, Lemma \ref{ureverse}, and \eqref{scalecondfirst}, we find
		\begin{equation}\label{nabuu}
		\int_{Q_{\sigma }(0)} \frac{|\nabla u|^p}{(1+ |u|)^{\gamma}} dx \leq C |\mu|(Q_{1}(0))  +  C \int_{Q_{1}(0)} (1+|u|)^{(p-1)\gamma}   dx\leq C,
		\end{equation}
		provided 	$\gamma\in (1,\min\{p,n/(n-p)\})$. Here $\sigma\in(0,1)$ and the constant $C$ may depend on $\sigma$.\\Also, by  Lemma \ref{test} and Sobolev embedding theorem  for Lebesgue spaces of mixed norm  (see \cite{BP, AB}), for    any $\alpha$ satisfying  $1-\frac{n-1}{p}<\alpha<1/p$,  $\alpha>0$, we have 
$$
		\|(1+|u|)^{\frac{p-\gamma}{p}} \phi\|^p_{L_{x'}^{\frac{(n-1)p}{n-1-(1-\alpha)p}}L_{x_n}^{\frac{p}{1-\alpha p}}} 
		\lesssim \int_\Om \phi^p d|\mu|  + \int_{\Omega} (1+|u|)^{(p-1)\gamma}  |\nabla \phi|^{p} dx,$$
		where $\phi\in C_0^\infty(Q_1(0))$, $\phi\geq0$, and $1<\gamma <\min\{p,n/(n-p)\}$. By Lemma \ref{ureverse} and \eqref{scalecondfirst}, this implies in particular that 
		\begin{equation} \label{u}
		\left\|1+|u|\right\|^{p-\gamma}_{L_{x'}^{\frac{(n-1)(p-\gamma)}{n-1-(1-\alpha)p}}L_{x_n}^{\frac{p-\gamma}{1-\alpha p}}(Q_{\sigma }(0))}\lesssim 1. 
		\end{equation}
		On the other hand, by H\"older's inequality we have 
		\begin{equation}\label{nab}
	\|\nabla u\|_{L^{p_2}_{x'}L_{x_n}^{p_1}(Q_{\sigma }(0))} \leq  \Big\| \frac{\nabla u}{ {(1+|u|)^{\frac{\gamma}{p}}}}\Big\|_{L^{p}(Q_{\sigma }(0))} \| 1+|u| \|^{\frac{\gamma}{p}}_{L^{\frac{\gamma q_2}{p}}_{x'}L_{x_n}^{\frac{\gamma q_1}{p}}(Q_{\sigma}(0))},
		\end{equation}
		where 
		$$\frac{1}{p_1}=\frac{1}{p} + \frac{1}{q_1}, \quad \frac{1}{p_2}=\frac{1}{p} + \frac{1}{q_2}, \quad p_1, p_2, q_1,q_2>0.$$
		We now choose 
		$$q_1=\frac{p(p-\gamma)}{\gamma(1-\alpha p)}, \qquad q_2=\frac{p(n-1)(p-\gamma)}{\gamma[n-1-(1-\alpha)p]},  $$	
		which forces 
		$$p_1=\frac{p-\gamma}{1-\alpha \gamma}, \qquad p_2=\frac{(n-1)(p-\gamma)}{n-1-(1-\alpha)\gamma}.$$
		At this point, we  combine \eqref{nabuu}, \eqref{u}, and \eqref{nab} to deduce
$$
		\|\nabla u\|_{L^{p_2}_{x'}L_{x_n}^{p_1}(Q_{\sigma }(0))}  \lesssim 1.$$
		This completes the proof of the lemma.
	\end{proof}\\\\
	We observe that as $\gamma\downarrow 1$, we have $p_1\uparrow \frac{p-1}{1-\alpha}$ and  $p_2\uparrow \frac{(n-1)(p-1)}{n-2+\alpha}$. Thus by H\"older's inequality we  
		obtain the following result.
	\begin{corollary}\label{co234}
		Let  $Q_r(x_0)\subset\Om$ and $1<p<n$, $n\geq 2$.  Then for any  $\sigma\in(0,1)$ and $q_1 >0$, we have 
$$
		\vertiii{\nabla u}_{L^{s_2}_{x'}L_{x_n}^{s_1}(Q_{\sigma r}(x_0))}\lesssim \left(\frac{|\mu|(Q_r(x_0))}{r^{n-1}}\right)^{\frac{1}{p-1}} + \frac{1}{r} \left(\fint_{Q_r(x_0)}|u|^{q_1} dx\right)^{\frac{1}{q_1}},$$
		 provided 
		$$\frac{p-1}{n-1}<s_1<p, \qquad 0<s_2< \frac{s_1(n-1)(p-1)}{s_1(n-1) -p+1}.$$
	\end{corollary}
	We now obtain a preliminary version of the   comparison estimate.
	
	\begin{lemma}\label{comsqrt3} Suppose that $n\geq 2$ and $Q_{r}(x_0)\Subset\Om$. Let  $u\in W_{\rm loc}^{1,p}(\Omega)$ be a solution of \eqref{quasi-measure} and let $w$ be as \eqref{eq1}. Then for $\kappa=(p-1)^2/2$ and $1<p<3/2$ we have  
		\begin{align*}
		&\left(	\fint_{Q_r(x_0)}|\nabla (u-w)|^{\ka}\right)^{\frac{1}{\kappa}} + \frac{1}{r} \left(	\fint_{Q_r(x_0)}|u-w|^{\ka}\right)^{\frac{1}{\kappa}}\\
		&\qquad\qquad\qquad \lesssim\left(\frac{|\mu|(Q_r(x_0))}{r^{n-1}}\right)^{\frac{1}{p-1}}+ \frac{|\mu|(Q_r(x_0))}{r^{n-1}}	
		\vertiii{\nabla u}^{2-p}_{L_{x'}^{\frac{(p-1)(2-p)}{3-p}}L^{2-p}_{x_n}(Q_r(x_0))}.
		\end{align*} 
\end{lemma}

	
	\begin{proof} Again by translating and scaling, we may assume that $x_0=(0,\dots,0)$, $Q_r(x_0)=(-1,1)^n$, and 
	\begin{equation}\label{scalecond}
		A:=|\mu|((-1,1)^n)^{\frac{1}{p-1}} + |\mu|((-1,1)^n) \|\nabla u\|^{2-p}_{L_{x'}^{\frac{(p-1)(2-p)}{3-p}}L^{2-p}_{x_n}((-1,1)^n)} \lesssim 1.
	\end{equation}		
	Moreover, we just need to show that 
	$$\|\nabla u-\nabla w\|_{L^\kappa((-1,1)^n)} + \|u- w\|_{L^\kappa((-1,1)^n)} \lesssim 1. $$
		For any given $k>0$ we now set 
$$E_k=(-1,1)^n\cap\{k<|u-w|<2k\},\quad \quad F_k=(-1,1)^n\cap\{|u-w|>k\}.$$
 We first recall that  for $1<p<2$,   
  
  \begin{equation}\label{gp}
  |\nabla u(x) -\nabla w(x)| \leq C g(x)^{1/p} + C g(x)^{1/2} |\nabla u(x)|^{\frac{2-p}{2}}, 
  \end{equation}
  where
  $$g(x):=\frac{|\nabla u(x) - \nabla w(x)|^2}{(|\nabla u(x)| + |\nabla w(x)|)^{2-p}}. $$
  
  Let $T_k$, $k>0,$ be the truncation operator defined in \eqref{Trun}. 
  Then it follows from \eqref{gp} that 
$$
		|\partial_{x_n} (T_{2k}-T_k)(u-w)(x',x_n)|\lesssim \left(\mathbf{1}_{E_k} g(x',x_n)^{\frac{1}{p}}+\mathbf{1}_{E_k}g(x',x_n)^{\frac{1}{2}}|\nabla u(x',x_n)|^{\frac{2-p}{2}}\right),$$
	where $\mathbf{1}_{ E_k}$ is the characteristic function of the set $E_k$.\\
	Thus, for any $0<\ga\leq 1$, by H\"older's inequality  we find
	\begin{align*}
		&\int_{(-1,1)^{n}} 	|(T_{2k}-T_k)(u-w)(x)|^\ga dx\\&\leq 2  \int_{(-1,1)^{n-1}} 	\left(\int_{-1}^{1}|\partial_{x_n}(T_{2k}-T_k)(u-w)(x',x_n)|dx_n\right)^\ga dx'\\&\lesssim
		\left(\int_{(-1,1)^n} \mathbf{1}_{E_k} g(x)^{\frac{1}{p}}dx\right)^\ga\\& \quad + \int_{(-1,1)^{n-1}} 	\left(\int_{-1}^{1}\mathbf{1}_{ E_k}g(x',x_n)dx_n\right)^\frac{\ga}{2} \left(\int_{-1}^{1}|\nabla u(x',x_n)|^{2-p}dx_n\right)^\frac{\ga}{2} dx'\\
		&\lesssim
		\left(\int_{E_k}  g(x)dx\right)^\frac{\ga}{p}  |E_k|^{\frac{\gamma (p-1)}{p}}\\	
	&\quad +  \left(\int_{E_k}  g(x)dx\right)^\frac{\ga}{2} \left(\int_{(-1,1)^{n-1}} 	\left(\int_{-1}^{1}|\nabla u(x',x_n)|^{2-p}dx_n\right)^\frac{\ga}{2-\ga} dx'\right)^{\frac{2-\ga}{2}}.
	\end{align*}

	On the other hand, using $T_{2k}(u-w)$ as a test function for \eqref{quasi-measure} we have 
	\begin{equation}\label{gEk}
	\int_{E_k} g(x) dx \leq \int_{(-1,1)^n\cap\{|u-w|< 2k\}} g(x) dx  \lesssim k |\mu|((-1,1)^n)\lesssim k A^{p-1}.
	\end{equation}

	Thus we get
	\begin{align*}
	&\int_{(-1,1)^{n}} 	|(T_{2k}-T_k)(u-w)(x)|^\ga dx	\lesssim k^{\frac{\ga }{p}} A^{\frac{\ga (p-1)}{p}} |E_k|^{\frac{\ga(p-1)}{p}}\\
	&+k^{\frac{\ga}{2}} |\mu|((-1,1)^n)^{\frac{\ga}{2}} \left(\int_{(-1,1)^{n-1}} 	\left(\int_{-1}^{1}|\nabla u(x',x_n)|^{2-p}dx_n\right)^\frac{\ga}{2-\ga} dx'\right)^{\frac{2-\ga}{2}}.
	\end{align*}
	We now set  	$$\gamma =p-1.$$
	Observe that $E_k\subset F_k$ and on $F_{2k}$ we have $|(T_{2k}-T_k)(u-w)|=k$. Then in view of \eqref{scalecond}, this gives	
$$
	k^\ga |F_{2k}|\lesssim k^{\frac{\ga}{p}} A^{\frac{\ga (p-1)}{p}} |F_k|^{\frac{\ga(p-1)}{p}} +k^{\frac{\ga}{2}} A^{\frac{p-1}{2}}.$$
	As $|F_{2k}|\leq |F_k|$ and $|F_{2k}|\leq 2^n$, we can write for any $\nu>0$, $$
	k^{\frac{\ga}{2}} |F_{2k}|^{1+\nu} \lesssim 
	k^{\frac{\ga}{p}-\frac{\ga}{2}} A^{\frac{\ga (p-1)}{p}} |F_k|^{\frac{\ga(p-1)}{p}+\nu}+ A^{\frac{p-1}{2}}.$$
	Then taking the supremum over $k>0$ we obtain  
$$
	||w-u||_{L^{\frac{\ga}{2(1+\nu)},\infty}}^{\frac{\ga}{2}}\lesssim 	||w-u||_{L^{\frac{\frac{\ga}{p}-\frac{\ga}{2}}{\frac{\ga(p-1)}{p}+\nu},\infty}}^{\frac{\ga}{p}-\frac{\ga}{2}}  A^{\frac{\ga (p-1)}{p}} + A^{\frac{p-1}{2}}.$$
	We now 	choose $\nu\geq 0$ such that 
$$
	\frac{\ga}{2(1+\nu)}=\frac{\frac{\ga}{p}-\frac{\ga}{2}}{\frac{\ga(p-1)}{p}+\nu} \Longleftrightarrow \nu=\frac{2-p-\ga(p-1)}{2(p-1)}=\frac{2-p-(p-1)^2}{2(p-1)}>0.$$
	Then using Young's inequality, we obtain
$$
	\|w-u\|_{L^{\frac{\ga(p-1)}{p-\ga(p-1)},\infty}} \lesssim A \lesssim 1, \qquad \gamma=p-1.$$
	Note that by \eqref{gEk} and Chebyshev's inequality,
	\begin{align*} 
	|\{ g>\lambda\}| & = |\{ g>\lambda\}\cap \{ |u-w|\geq k\}| + |\{ g>\lambda\}\cap \{ |u-w|< k\}|\\
	 & \lesssim k^{-\frac{\ga(p-1)}{p-\ga(p-1)}}||u-w||_{L_\omega^{\frac{\ga(p-1)}{p-\ga(p-1)},\infty}}^{\frac{\ga(p-1)}{p-\ga(p-1)}}+\frac{1}{\lambda}\int_{(-1,1)^n\cap\{|u-w|< k\}}g dx
	\\&\lesssim k^{-\frac{\ga(p-1)}{p-\ga(p-1)}} +\frac{ k}{\lambda} |\mu|((-1,1)^n).
	\end{align*} 
	Thus choosing $k=\lam^{1-\frac{\ga(p-1)}{p}} |\mu|((-1,1)^n)^{-1+\frac{\ga(p-1)}{p}}$,
	one has
$$
		|\{g>\lambda\}|\lesssim \lam^{-\frac{\ga(p-1)}{p}} |\mu|((-1,1)^n)^{\frac{\ga(p-1)}{p}} \lesssim \lambda^{-\frac{(p-1)^2}{p}} |\mu|((-1,1)^n)^{\frac{(p-1)^2}{p}}.$$
		It follows that  for any $\beta<\frac{(p-1)^2}{p}$,
$$
		\|g\|_{L^{\beta}}\lesssim |\mu|((-1,1)^n),$$
which in view of \eqref{gp} and \eqref{scalecond} yields
$$\|\nabla (u-w)\|_{L^{p\beta}}\lesssim  	\|g\|_{L^{\beta}}^{\frac{1}{p}}+ 	\|g\|_{L^{\beta}}^{\frac{1}{2}} \|\nabla u\|_{L^{p\beta}}^{\frac{2-p}{2}} \lesssim A^{\frac{p-1}{p}}+ |\mu|((-1,1)^n)^{\frac{1}{2}} \|\nabla u\|_{L^{p\beta}}^{\frac{2-p}{2}}.
$$
Next, we choose	$	\beta=\frac{(p-1)^2}{2p},$
	and employ Holder's inequality to deduce 
	\begin{equation*}
		\|\nabla u\|_{L^{p\beta}} = \|\nabla u\|_{L^{\frac{(p-1)^2}{2}}} \lesssim	\|\nabla u\|_{L_{x'}^{\frac{(p-1)(2-p)}{3-p}}L^{2-p}_{x_n}(Q_1(0))}.
	\end{equation*}
	Here we used the fact that 	$\frac{(p-1)^2}{2}< \frac{(p-1)(2-p)}{3-p}$ provided $ 1<p<3/2$.\\
	These inequalities now give 
$$
		||\nabla (u-w)||_{L^{\frac{(p-1)^2}{2}}}\lesssim A^{\frac{p-1}{p}}+	 |\mu|((-1,1)^n)^{\frac{1}{2}} \|\nabla u\|_{L_{x'}^{\frac{(p-1)(2-p)}{3-p}}L^{2-p}_{x_n}(Q_1(0))}^{\frac{2-p}{2}}\lesssim 1,$$
	which completes the proof of the lemma.
\end{proof}	\\\\
Using Lemma \ref{comsqrt3}, we next prove a reverse H\"older type inequality for the gradient.
\begin{lemma}\label{revholderformixgrad} Suppose that  $u$ is a $W^{1,p}_{\rm loc}(\Om)$ solution of \eqref{quasi-measure}.
	 Let $Q_r(x_0)\subset\Om$, $n\geq 2$, and $\kappa=(p-1)^2/2$, where $1<p< 3/2$.   Let  
	$$\theta\in\left(0,\frac{2\kappa(p-1)}{(2-p)(p-\kappa)}\right)$$ and define $s_1, s_2$  by the equations
	\begin{equation}\label{defequ}
	\frac{1}{2-p}=\frac{\theta}{\kappa} + \frac{1-\theta}{s_1}, \quad \frac{3-p}{(p-1)(2-p)}=\frac{\theta}{\kappa} + \frac{1-\theta}{s_2}.
	\end{equation}
	Then,
	\begin{equation}\label{pbound}
	2-p<s_1 <p,  \qquad s_1> s_2>\frac{(p-1)(2-p)}{3-p},
	\end{equation}
	\begin{equation}\label{p2cond}
	s_2 < \frac{s_1(n-1)(p-1)}{s_1(n-1) -p+1},
	\end{equation}
and 	moreover,

	\begin{align}\label{rev}
	 \vertiii{\nabla u}_{L^{s_2}_{x'}L_{x_n}^{s_1}(Q_{\sigma r}(x_0))} \lesssim   \left(\frac{|\mu|(Q_r(x_0))}{r^{n-1}}\right)^{\frac{1}{p-1}}+ \left(\fint_{Q_r(x_0)} |\nabla u|^\kappa dx\right)^{\frac{1}{\kappa}},
	\end{align}
	provided $\sigma\in(0,1)$.
\end{lemma}

\begin{remark}\label{ktoe}
	By H\"older's inequality,  \eqref{rev} also holds with $$s_1=s_2=\frac{(p-1)(2-p)}{3-p}>\kappa$$ and then a covering/iteration argument as in \cite[Remark 6.12]{Giu} implies that we can replace $\kappa$ with any $\epsilon>0$ in \eqref{rev}.
\end{remark}
\begin{proof} The proof of \eqref{pbound} and \eqref{p2cond} is obvious.
	In order to show \eqref{rev}, we first show that 
	\begin{align}\label{r2r} \vertiii{\nabla u}_{L^{s_2}_{x'}L_{x_n}^{s_1}(Q_{ r/2}(x_0))} 
		 \lesssim \left(\frac{|\mu|(Q_r(x_0))}{r^{n-1}}\right)^{\frac{1}{p-1}}+ \vertiii{\nabla u}_{L_{x'}^{\frac{(p-1)(2-p)}{3-p}}L^{2-p}_{x_n}(Q_r(x_0))}
	\end{align}
	for any $Q_r(x_0)\Subset\Om$. Moreover, by scaling, to prove \eqref{r2r} we may assume that $Q_r(x_0)=Q_1(0)=(-1,1)^n$, 
	\begin{equation}\label{scalco2}
     |\mu|((-1,1)^n)^{\frac{1}{p-1}} +  \|\nabla u\|_{L_{x'}^{\frac{(p-1)(2-p)}{3-p}}L^{2-p}_{x_n}((-1,1)^n)} \lesssim 1,
	\end{equation}
	and prove that 
	\begin{equation}\label{goal1}
	\|\nabla u\|_{L^{s_2}_{x'}L_{x_n}^{s_1}((-1/2,1/2)^n)} \lesssim 1.
	\end{equation}
	Let $w$ be as in \eqref{eq1} with $Q_r(x_0)=(-1,1)^n$.  Then, by Lemma \ref{comsqrt3}  and \eqref{scalco2},
$$\left(\int_{Q_1(0)}|u-w|^{\kappa}\right)^{\frac{1}{\kappa}}  \lesssim 1, \qquad \kappa=\frac{(p-1)^2}{2}.$$
Thus for  $\lambda=\fint_{Q_{7/8}(_0)} w dx$, by Poincar\'e inequality we have 
$$\left(	\int_{Q_{7/8}(0)}|u-\lambda|^{\kappa}\right)^{\frac{1}{\kappa}}  \lesssim  \left(	\int_{Q_{7/8}(0)}|w-\lambda|^{\kappa}\right)^{\frac{1}{\kappa}} +1 \lesssim 	\int_{Q_{7/8}(0)}|\nabla w|dx  + 1.$$
Then by the reverse H\"older property of $\nabla w$ and Lemma \ref{comsqrt3} we can now deduce
\begin{align*}
\left(	\int_{Q_{7/8}(0)}|u-\lambda|^{\kappa}\right)^{\frac{1}{\kappa}}
& \lesssim 	\left(\int_{Q_{1}(0)}|\nabla w|^\kappa dx\right)^{\frac{1}{\kappa}}  + 1 \lesssim 	\left(\int_{Q_{1}(0)}|\nabla u|^\kappa dx\right)^{\frac{1}{\kappa}}  + 1\\
& \lesssim  \|\nabla u\|_{L_{x'}^{\frac{(p-1)(2-p)}{3-p}}L^{2-p}_{x_n}(Q_1(0))} +1 \lesssim 1.
\end{align*}		
 On the other hand, it follows from Corollary \ref{co234} that  
$$\|\nabla u\|_{L^{s_2}_{x'}L_{x_n}^{s_1}(Q_{1/2}(0))}  \lesssim   1+ \left(\int_{Q_{7/8}(0)}|u-\lambda|^{\kappa} dx \right)^{\frac{1}{\kappa}}	$$
for any $\lambda\in\RR$. 
Thus we obtain \eqref{goal1} as desired.

Next, for any cube $Q_r(x_0)\subset\Om$ we consider the cubes $Q_t(x_0)\subset Q_s(x_0)$ where $0<t<s< r$. Recall that $Q_t(x_0)=\{x_{0}'+(-t, t)^{n-1}\}\times \{x_{0n} + (-t, t) \}$.
We can cover the interval $I_t(x_{0n})= x_{0n} + (-t, t)$
by a sequence of intervals $I_i= y_{in} + (-(s-t)/2, (s-t)/2) $, $y_{in}\in I_t(x_{0n})$, in such a way
that any point $y\in\RR$  belongs to almost 3 intervals of the collection $\{2 I_i\}=\{  y_{in} + (-(s-t), (s-t))\}$.
Note that we have $1\leq i\leq M$, where $M \lesssim t/(s-t)$.

Likewise, we can cover the $(n-1)$-dimensional cube $J_t(x_{0}')= x_{0}'+ (-t, t)^{n-1}$
by a sequence of $(n-1)$-dimensional cubes $J_j= y_{j}' +(-(s-t)/2, (s-t)/2)^{n-1} $, $y_{j}'\in J_t(x_{0}')$, in such a way
that any point $y\in\RR^{n-1}$  belongs to almost $N(n)$ cubes of the collection $\{2 J_j\}=\{y_j' + (-(s-t), (s-t))^{n-1}\}$. Also, $1\leq j\leq  M'$, where 
$M'\lesssim [t/(s-t)]^{n-1}$.\\
Note that we have $2J_j\times 2 I_i \subset Q_s(x_0)$ for any $i,j$.\\
Then applying \eqref{r2r}, we get 
\begin{equation}\label{s1s2}
 \vertiii{\nabla u}_{L^{s_2}_{x'}L_{x_n}^{s_1}(J_j\times I_i)}  \lesssim \left(\frac{|\mu|(Q_r(x_0))}{(s-t)^{n-1}}\right)^{\frac{1}{p-1}}  +  \vertiii{\nabla u}_{L^{\frac{(p-1)(2-p)}{3-p}}_{x'}L_{x_n}^{2-p}(2J_j\times 2I_i)}.
\end{equation}	
We now observe that, as $s_2<s_1$, 
$$
\|\nabla u\|_{L^{s_2}_{x'}L_{x_n}^{s_1}(Q_{t}(x_0))}^{s_2}\leq \sum_{i,j}  \int_{J_j} \left(\int_{I_i} |\nabla u|^{s_1} dx_n\right)^{\frac{s_2}{s_1}} dx'= \sum_{i,j}  \|\nabla u\|_{L^{s_2}_{x'}L_{x_n}^{s_1}(J_j\times I_i)}^{s_2}.$$
Thus in view of \eqref{s1s2} we find 
\begin{align*}
 \|\nabla u\|_{L^{s_2}_{x'}L_{x_n}^{s_1}(Q_{t}(x_0))}^{s_2}&\lesssim  r^{s_2[\frac{1}{s_1} + \frac{n-1}{s_2}]} \sum_{i,j} \left(\frac{|\mu|(Q_r(x_0))}{(s-t)^{n-1}}\right)^{\frac{s_2}{p-1}} \\
&+  r^{s_2[\frac{1}{s_1} + \frac{n-1}{s_2}]} (s-t)^{-s_2\Big[\frac{1}{2-p} + \frac{(n-1)(3-p)}{(2-p)(p-1)}\Big]} \sum_{i,j}  \|\nabla u\|^{s_2}_{L^{\frac{(p-1)(2-p)}{3-p}}_{x'}L_{x_n}^{2-p}(2J_j\times 2I_i)}.
\end{align*}	

Since  $s_2> (p-1)(2-p)/(3-p)$ and $(p-1)/(3-p)<1$, we then have
\begin{align*}
& \|\nabla u\|_{L^{s_2}_{x'}L_{x_n}^{s_1}(Q_{t}(x_0))}^{s_2}\lesssim   \left(\frac{r}{s-t}\right)^{n}   r^{s_2[\frac{1}{s_1} + \frac{n-1}{s_2}]} \left(\frac{|\mu|(Q_r(x_0))}{(s-t)^{n-1}}\right)^{\frac{s_2}{p-1}}\\
&\qquad  + \left(\frac{r}{s-t}\right)^{\frac{2s_2}{p-1}}  r^{s_2[\frac{1}{s_1} + \frac{n-1}{s_2}]} (s-t)^{-s_2\Big[\frac{1}{2-p} + \frac{(n-1)(3-p)}{(2-p)(p-1)}\Big]} \times\\
&\qquad \qquad\times \left(\sum_{j} \int_{2J_j} \left(\sum_{i} \int_{2I_i} |\nabla u|^{2-p} dx_n\right)^{\frac{p-1}{3-p}} dx' \right)^{\frac{s_2(3-p)}{(p-1)(2-p)}}.	 
\end{align*}	This gives 
\begin{align*}
 &\|\nabla u\|_{L^{s_2}_{x'}L_{x_n}^{s_1}(Q_{t}(x_0))}^{s_2} 
\lesssim \left(\frac{r}{s-t}\right)^{n} r^{s_2[\frac{1}{s_1} + \frac{n-1}{s_2}]} \left(\frac{|\mu|(Q_r(x_0))}{(s-t)^{n-1}}\right)^{\frac{s_2}{p-1}}\\
& \quad\quad+  \left(\frac{r}{s-t}\right)^{\frac{2s_2}{p-1}}   r^{s_2[\frac{1}{s_1} + \frac{n-1}{s_2}]} (s-t)^{-s_2\Big[\frac{1}{2-p} + \frac{(n-1)(3-p)}{(2-p)(p-1)}\Big]} \|\nabla u\|^{s_2}_{L^{\frac{(p-1)(2-p)}{3-p}}_{x'}L_{x_n}^{2-p}(Q_s(x_0))},
\end{align*}
and thus by H\"older's inequality,
\begin{align*}
&\|\nabla u\|_{L^{s_2}_{x'}L_{x_n}^{s_1}(Q_{t}(x_0))}^{s_2} 
\lesssim \left(\frac{r}{s-t}\right)^{n} r^{s_2[\frac{1}{s_1} + \frac{n-1}{s_2}]}  \left(\frac{|\mu|(Q_r(x_0))}{(s-t)^{n-1}}\right)^{\frac{s_2}{p-1}} \\
& \qquad +  \left(\frac{r}{s-t}\right)^{\frac{2s_2}{p-1}} r^{s_2[\frac{1}{s_1} + \frac{n-1}{s_2}]} (s-t)^{-s_2\Big[\frac{1}{2-p} + \frac{(n-1)(3-p)}{(2-p)(p-1)}\Big]}\|\nabla u\|^{\theta s_2}_{L^\kappa} \|\nabla u\|^{(1-\theta)s_2}_{L^{s_2}_{x'}L_{x_n}^{s_1}(Q_s(x_0))}.
\end{align*}	
Then by Young's inequality, we get
\begin{align*}
&\|\nabla u\|_{L^{s_2}_{x'}L_{x_n}^{s_1}(Q_{t}(x_0))}^{s_2} 
\leq C \left(\frac{r}{s-t}\right)^{n} r^{s_2[\frac{1}{s_1} + \frac{n-1}{s_2}]} \left(\frac{|\mu|(Q_r(x_0))}{(s-t)^{n-1}}\right)^{\frac{s_2}{p-1}} + \frac{1}{2} \|\nabla u\|^{s_2}_{L^{s_2}_{x'}L_{x_n}^{s_1}(Q_s(x_0))}\\
&\qquad  + C \left(\frac{r}{s-t}\right)^{\frac{2s_2}{(p-1)\theta}}  r^{\frac{s_2}{\theta}[\frac{1}{s_1} + \frac{n-1}{s_2}]} (s-t)^{-\frac{s_2}{\theta}\Big[\frac{1}{2-p} + \frac{(n-1)(3-p)}{(2-p)(p-1)}\Big]} \|\nabla u\|^{ s_2}_{L^\kappa(Q_s(x_0))}. 
\end{align*}	

This enables us to apply   \cite[Lemma 6.1]{Giu} and use \eqref{defequ} to obtain 
\begin{align*}
&\|\nabla u\|^{s_2}_{L^{s_2}_{x'}L_{x_n}^{s_1}(Q_{\sigma r}(x_0))} 
\lesssim r^{s_2[\frac{1}{s_1} + \frac{n-1}{s_2}]} \left(\frac{|\mu|(Q_r(x_0))}{r^{n-1}}\right)^{\frac{s_2}{p-1}} \\
 & \qquad + r^{\frac{s_2}{\theta}[\frac{1}{s_1} + \frac{n-1}{s_2}]} r^{-\frac{s_2}{\theta}\Big[\frac{1}{2-p} + \frac{(n-1)(3-p)}{(2-p)(p-1)}\Big]} \|\nabla u\|^{s_2}_{L^\kappa(Q_r(x_0))}\\
 &\qquad \lesssim r^{s_2[\frac{1}{s_1} + \frac{n-1}{s_2}]} \left(\frac{|\mu|(Q_r(x_0))}{r^{n-1}}\right)^{\frac{s_2}{p-1}} + r^{s_2[\frac{1}{s_1} + \frac{n-1}{s_2}]}  r^{\frac{-s_2 n}{\kappa}} \|\nabla u\|^{s_2}_{L^\kappa(Q_r(x_0))}.
\end{align*}	
This completes the proof of the lemma.
\end{proof}\\\\
We mention again that Theorem \ref{maininterior} follows from Lemmas \ref{comsqrt3},  \ref{revholderformixgrad}, and Corollary \ref{co234}. 
We now describe the boundary version of 
Theorem \ref{maininterior}. Let $u$ be a $W^{1,p}_{0}(\Om)$ solution of \eqref{quasi-measure}, and let 
$x_0\in \partial\Om$, $r< {\rm diam}(\Om)/10$. We then extend both $u$ and $\mu$ by zero outside $\Om$ and 
 consider the the unique solution $w\in W^{1,p}(Q_r(x_0))+u$ to
\begin{equation}\label{eq2}
\left\{ \begin{array}{rcl}
- \operatorname{div}\left( {A(x,\nabla w)} \right) &=& 0 \quad \text{in} \quad Q_r(x_0)\cap \Omega, \\ 
w &=& u\quad \text{on} \quad \partial \left(Q_r(x_0)\cap \Omega\right).  
\end{array} \right.
\end{equation}
Then we have the following boundary counterpart of Theorem \ref{maininterior}.

\begin{theorem}\label{mainboundary} Let $x_0\in\partial\Om$, $u\in W^{1,p}_0(\Om)$ and $w$ be as in \eqref{eq2}. Then with   $\kappa=(p-1)^2/2$ and $1<p < 3/2$, for any 
$\Sigma \in (1,2]$	we have   
	\begin{align*}
	&\left(	\fint_{Q_r(x_0)}|\nabla (u-w)|^{\ka}\right)^{\frac{1}{\kappa}} + \frac{1}{r}\left(	\fint_{Q_r(x_0)}|u-w|^{\ka}\right)^{\frac{1}{\kappa}}\\
	&\qquad \lesssim \left(\frac{|\mu|(Q_{\Sigma r}(x_0))}{r^{n-1}}\right)^{\frac{1}{p-1}} + \frac{|\mu|(Q_{\Sigma r}(x_0))}{r^{n-1}}	
	\left(	\fint_{Q_{\Sigma r}(x_0)}|\nabla u|^{\ka}\right)^{\frac{2-p}{\kappa}},
	\end{align*} 
	and 
	\begin{align*}
	\left(	\fint_{Q_r(x_0)}|u-w|^{\ka}\right)^{\frac{1}{\kappa}} &\lesssim\left(\frac{|\mu|(Q_{\Sigma r}(x_0))}{r^{n-p}}\right)^{\frac{1}{p-1}}+\frac{|\mu|(Q_{\Sigma r}(x_0))}{r^{n-p}}	
	\left(	\fint_{Q_{\Sigma r}(x_0)}|u|^{\ka}\right)^{\frac{2-p}{\kappa}}.
	\end{align*}
\end{theorem}
Note that if $u\in W^{1,p}_0(\Om)$ then in Lemma \ref{test} we may take $\phi\in C_0^\infty(\RR^n)$ (i.e.,  $\phi$ does not need to have compact support in $\Om$). With this observation, we see that Theorem \ref{mainboundary} can be proved in the same manner as Theorem \ref{maininterior}.\\
We now devote the rest of this section to the proof of Corollary \ref{poicaresub}.

\begin{proof}[Proof of Corollary \ref{poicaresub}] By Remark \ref{ktoe}, it is enough to consider the case $\ep=\kappa$.
	Let  $w$ be as in \eqref{eq1} with $Q_{\Sigma_0 r}(x_0)$ in place of $Q_{r}(x_0)$, where we choose $\Sigma_0$ so that $1<\Sigma_0<\Sigma\leq 2$.  By $L^1$ Poincar\'e inequality,  Lemma  \ref{comsqrt3}, and Young's inequality
	we find 
	\begin{align*}
	&\left(	\inf_{q\in \RR}\fint_{Q_r(x_0)}|u-q|^{\ka}\right)^{\frac{1}{\kappa}} \lesssim \left(\inf_{q\in \RR}\fint_{Q_r(x_0)}|w-q|^{\ka}\right)^{\frac{1}{\kappa}}
	+ \left(	\fint_{Q_r(x_0)}|u-w|^{\ka}\right)^{\frac{1}{\kappa}}\\
	&\qquad \lesssim	r \fint_{Q_r(x_0)}|\nabla w|dx + \left(\frac{|\mu|(B_{\Sigma_0 r})(x_0)}{r^{n-p}}\right)^{\frac{1}{p-1}} +
	r \vertiii{\nabla u}_{L_{x'}^{\frac{(p-1)(2-p)}{3-p}}L^{2-p}_{x_n}(Q_{\Sigma_0 r}(x_0))}.
	\end{align*}
	We now  use the reverse H\"older property of $\nabla w$ to obtain	from the above inequality that 
	\begin{align*}
	\left(	\inf_{q\in \RR}\fint_{Q_r(x_0)}|u-q|^{\ka}\right)^{\frac{1}{\kappa}} 
&\lesssim	r \left(\fint_{Q_{\Sigma_0 r}(x_0)}|\nabla u- \nabla w|^\kappa dx\right)^{\frac{1}{\kappa}}+r \left(	\fint_{Q_{\Sigma_0 r}(x_0)}|\nabla u|^{\ka}\right)^{\frac{1}{\kappa}}\\
&+ \left(\frac{|\mu|(B_{\Sigma_0 r})(x_0)}{r^{n-p}}\right)^{\frac{1}{p-1}} +
r \vertiii{\nabla u}_{L_{x'}^{\frac{(p-1)(2-p)}{3-p}}L^{2-p}_{x_n}(Q_{\Sigma_0 r}(x_0))}.
	\end{align*}
At this point, we use 	Lemmas  \ref{comsqrt3} and \ref{revholderformixgrad}  to conclude the proof.
\end{proof}

	\section{Proof of Theorem \ref{upw}}

To prove Theorem \ref{upw}, we shall need the following sharp quantitative regularity/decay estimates for homogeneous equations.	
\begin{lemma}\label{osckkk}
	Under \eqref{condi1}--\eqref{condi2}, let $w\in W^{1,p}(\Om)$ be a solution of $\operatorname{div}\left( {A(x,\nabla w)} \right) = 0$ in $\Om$. Then there exists 
	$\alpha_0\in (0,1]$ such that  for any $Q_\rho(x_0)\subset Q_R(x_0)\subset\Om$, and   $\ep\in (0,1)$, we have 
	\begin{equation}\label{oscp}
	\fint_{Q_\rho(x_0)} |w-(w)_{Q_\rho(x_0)}|^p dx \lesssim \left(\frac{\rho}{R}\right)^{\alpha_0 p} \fint_{Q_R(x_0)} |w-(w)_{Q_R(x_0)}|^p dx,
	\end{equation}  
	and 
	\begin{equation}\label{oscqk}
	\inf_{q\in \RR}\fint_{Q_\rho(x_0)} |w-q|^\ep dx \lesssim \left(\frac{\rho}{R}\right)^{\alpha_0 \ep} \inf_{q\in\RR}\fint_{Q_R(x_0)} |w-q|^\ep dx.
	\end{equation} 
\end{lemma}

\begin{proof}
	The proof of \eqref{oscp} follows from \cite[Chapter 7]{Giu}, whereas the proof of \eqref{oscqk} follows from \eqref{oscp} and the reverse H\"older property of $w$.
\end{proof}

We are now ready for the proof of Theorem \ref{upw}.

\begin{proof}[Proof of \eqref{upw}]
We first observe that for each cube $Q_\rho(x_0)\Subset\Om$ and $f\in L^\kappa_{\rm loc}(\Om)$, there exists $q_{\rho,x_0}=q_{\rho, x_0}(f)\in\RR$ such that 
	$$\inf_{q\in \RR} \left(\fint_{Q_\rho(x_0)} |f-q|^\kappa dx\right)^{\frac{1}{\kappa}}=\left(\fint_{Q_\rho(x_0)} |f-q_{\rho, x_0}|^\kappa dx\right)^{\frac{1}{\kappa}}.$$
	Then it is known that (see, e.g., \cite[Lemma 4.1]{DS})
	\begin{equation}\label{eLP}
	\lim_{\rho\rightarrow 0} q_{\rho,x_0}(f)=f(x_0) \quad \text{  a.e. with respect to } x_0\in\RR^n.
	\end{equation}
	For a cube $Q_\rho=Q_\rho(x_0)\subset\Omega$, we  now define
	$$
	\mathbf{J}(\rho, x_0)=\mathbf{J}(\rho, x_0,u):=	\inf_{q\in \RR} \left(\fint_{Q_\rho(x_0)} |u-q|^\kappa dx\right)^{\frac{1}{\kappa}}.
	$$
	This kind of quantity is a good substitution for the mean oscillation for functions that may not belong to $L^1_{\rm loc}$ (see, e.g., \cite[Section 4]{DS}).  \\
	Next, for any $\varepsilon\in (0,1)$ and $Q_{r}(x_0)\Subset\Omega$
	by Lemma \ref{osckkk} and quasi-triangle inequality,  we can find $\alpha_0\in (0,1]$ such that 
	\begin{align}\nonumber
	\mathbf{J}(\varepsilon r/2, x_0) &\leq  \left(\fint_{Q_{\varepsilon r/2}(x_0)} |u-q_{\varepsilon r/2, x_0}(w)|^\kappa dx\right)^{\frac{1}{\kappa}}\\
	&\leq C \left(\fint_{Q_{\varepsilon r/2}(x_0)} |w-q_{\varepsilon r/2, x_0}(w)|^\kappa dx\right)^{\frac{1}{\kappa}}+C	\left(\fint_{Q_{\varepsilon r/2}(x_0)} |u-w|^\kappa dx\right)^{\frac{1}{\kappa}}\nonumber\\
	&\leq
	C	\varepsilon^{\alpha_0} \left(\fint_{Q_{r/2}(x_0)} |w-q_{r, x_0}(w)|^\kappa dx\right)^{\frac{1}{\kappa}}+C\varepsilon^{-n/\kappa}	\left(\fint_{Q_{ r/2}(x_0)} |u-w|^\kappa dx\right)^{\frac{1}{\kappa}}.\nonumber	
	\end{align}
	Here we choose $w$ as in 	\eqref{eq1} with $Q_{r/2}(x_0)$ in place of $Q_{r}(x_0)$.\\
	Thus after similar manipulations we get  	 
	\begin{align*}
	\mathbf{J}(\varepsilon r/2, x_0)\leq C\varepsilon^{\alpha_0}\mathbf{J}(r/2, x_0)+ C_\varepsilon	\left(\fint_{Q_{ r/2}(x_0)} |u-w|^\kappa dx\right)^{\frac{1}{\kappa}}.
	\end{align*}
	We now apply  Theorem \ref{maininterior} to bound the second term on the right-hand side of the above inequality. This yields that 
	$$
	\mathbf{J}(\varepsilon r/2, x_0)\leq C\varepsilon^{\alpha_0}\mathbf{J}(r/2, x_0)+C_\varepsilon\left(\frac{|\mu|(Q_{r}(x_0))}{r^{n-p}}\right)^{\frac{1}{p-1}}+C_\varepsilon \frac{|\mu|(Q_{r}(x_0))}{r^{n-p}} \mathbf{J}(r, x_0)^{2-p}.$$
	Then by Young's inequality we find
	\begin{align}\label{intA}
	\mathbf{J}(\varepsilon r/2, x_0)&\leq C\varepsilon^{\alpha_0}\mathbf{J}(r, x_0)+C_\varepsilon\left(\frac{|\mu|(Q_{r}(x_0))}{r^{n-p}}\right)^{\frac{1}{p-1}}.
	\end{align}
	We now choose $\varepsilon <\frac{1}{4\sqrt{n}}$ small enough so that $
	C(\varepsilon)^{\alpha_0}\leq \frac{1}{4}$, where $C$ is the constant in \eqref{intA}.
	Set $R_j=(\varepsilon/2)^{j} R$, $Q_j:=Q_{R_j}(x_0)$.  Applying \eqref{intA}  yields 
	$$
	\mathbf{J}(R_{j+1}, x_0) \leq \frac{1}{4} \mathbf{J}(R_j, x_0) +C\left( \frac{|\mu|(Q_{j})}{R_{j}^{n-p}}\right)^{\frac{1}{p-1}}. 
	$$
	Summing this up over $j\in \{2,3,...,m-1\}$, we obtain 
	\begin{align}\label{z1}
	\sum_{j=2}^{m}\mathbf{J}(R_j, x_0) &\leq C\, \mathbf{J}(R_{2}, x_0) +C\sum_{j=2}^{m-1}\left(\frac{|\mu|(Q_{j})}{R_{j}^{n-p}}\right)^{\frac{1}{p-1}}.
	\end{align}
	It is not hard to see that 
	$$|q_{R_{j+1}, x_0}(u)-q_{R_{j}, x_0}(u)| \leq C (\mathbf{J}(R_{j+1}, x_0)+\mathbf{J}(R_j, x_0)) $$
	for all $j\geq 1$, and then by iterating we find
	$$|q_{R_{m}, x_0}(u)-q_{R_{2}, x_0}(u)| \leq C \sum_{j=2}^{m}\mathbf{J}(R_j, x_0). $$
	Note also that 
$$
	\mathbf{J}(R_{2}, x_0) +  |q_{R_{2}, x_0}(u)| \leq C \mathbf{J}(R_{2}, x_0) +  C \left(\fint_{Q_{2}} |u|^\kappa dx\right)^{\frac{1}{\kappa}}\leq C \left(\fint_{B_{R}(x_0)} |u|^\kappa dx\right)^{\frac{1}{\kappa}}.$$
	Thus using these in   \eqref{z1} we get 
	\begin{align*}
	|q_{R_{m}, x_0}(u)| &  \leq  C\, \mathbf{J}(R_{2}, x_0) + C |q_{R_{2}, x_0}(u)| +C\sum_{j=2}^{m-1}\left(\frac{|\mu|(Q_{j})}{R_{j}^{n-p}}\right)^{\frac{1}{p-1}}\nonumber\\
	&\leq C  \left(\fint_{B_{R}(x_0)} |u|^\kappa dx\right)^{\frac{1}{\kappa}} + C \sum_{j=2}^{m-1}\left(\frac{|\mu|(Q_{j})}{R_{j}^{n-p}}\right)^{\frac{1}{p-1}}.
	\end{align*}
	Note that 
$$
	\sum_{j=2}^{m-1}\left(\frac{|\mu|(Q_{j})}{R_{j}^{n-p}}\right)^{\frac{1}{p-1}} \leq  \sum_{j=2}^{m-1}\left(\frac{|\mu|(B_{\sqrt{n}R_j}(x_0))}{R_{j}^{n-p}}\right)^{\frac{1}{p-1}}
	\leq  C\, {\bf W}^{R}_{1,p}(|\mu|)(x_0).$$
	Thus in view of \eqref{eLP} we obtain \eqref{Wpointwiseu} as desired.
\end{proof}

	\section{Proof of Theorem \ref{main-Ric}}
	
	In this section, we provide a proof of Theorem \ref{main-Ric}.
	
		
	\noindent 	Step 1: In this step we  assume that $\mu\in L^\infty(\Om)$. 	Set 
		\begin{equation*}
		E(\mu):=\left\{ u \in W^{1,p}_0(\Om): |\nabla u(x)|\leq 2C_0\mathbf{P}[\mu](x) {\rm ~ a.e.}\right\},
		\end{equation*}
		where  $C_0$ is the constant in \eqref{Wpointwiseu"'}, and  
		\begin{equation*}
		\mathbf{P}[\mu](x):=	\Big[ {\bf I}_{1}^{2{\rm diam}(\Om)}(|\mu|)(x)\Big]^{\frac{1}{p-1}}, \qquad x\in \RR^n.
		\end{equation*}
		Here we extend $u$ by zero outside $\Om$.
		Note that by  inequality (2.10)  of \cite{Ph1} and \eqref{capcondi}, one has 
		\begin{equation}\label{zz}
		\mathbf{P}[\mathbf{P}[\mu]^q](x)\leq C_1 c_0^{\frac{q-p+1}{(p-1)^2}}\mathbf{P}[\mu](x) \qquad\forall x\in \mathbb{R}^n. 
		\end{equation}
		
		Now let $S: E(\mu)\rightarrow W_0^{1,p}(\Om)$ be defined by 
		$S(v)=u$ where $u\in W_0^{1,p}(\Om)$ is the unique renormalized  solution of
		\begin{eqnarray*}
			\left\{\begin{array}{rcl}
				-{\rm div}(A(x, \nabla u)) &=& |\nabla v|^q + \mu \quad {\rm in}~ \Om,\\
				u&=&0\quad ~~~~~~~~~~ {\rm on~}\partial\Om.
			\end{array}
			\right.
		\end{eqnarray*}
		
		By \eqref{Wpointwiseu"'} and \eqref{zz}, one obtains
		\begin{align*}
		|\nabla u(x)|&\leq C_0 	\mathbf{P}[ |\nabla u|^{q}+|\mu|](x) \leq C_0\left( (2C_0)^q\mathbf{P}[  \mathbf{P}[\mu]^q](x)+	\mathbf{P}[|\mu|](x)\right)\\
		&\leq C_0\left( (2C_0)^qC_1 c_0^{\frac{q-p+1}{(p-1)^2}}+	1\right)\mathbf{P}[|\mu|](x) \leq 
		2C_0 \mathbf{P}[|\mu|](x)
		\end{align*}
		provided  $c_0>0$ is  small enough. This means that $S(E(\mu))\subset E(\mu)$. Then, similar to \cite[Proof of Theorem 1.9]{HPanna}, we obtain that $S$ has  a fixed point  in  $E(\mu)$. Thus, there exists a solution $u\in E(\mu)$ to 
		\begin{eqnarray}\label{zzz}
			\left\{\begin{array}{rcl}
				-{\rm div}(A(x, \nabla u)) &=& |\nabla u|^q + \mu \quad {\rm in}~ \Om,\\
				u&=&0\quad ~~~~~~~~~~ {\rm on~}\partial\Om.
			\end{array}
			\right.
		\end{eqnarray}
		Step 2: We now extend $\mu$ be zero out side $\Om$ and let $\mu_k=\rho_k*\mu$, where ${\{\rho_k\}}_{k>0}$ is a standard
		sequence of mollifiers: $\rho_k(\cdot):= k^n\rho(k\cdot)$ for a nonnegative and radial function $\rho\in C_0^\infty(B_1(0))$ such that  $\int_{\RR^n} \rho dx=1$. It is not hard to see that there exists a constant $B>0$ such that 
		$$|\mu_k|(K) \leq  B c_0 \, {\rm Cap}_{1,\, \frac{q}{q-p+1}}(K)$$
		for any compact set $K\subset\Om$; see \cite[Lemma 5.7]{Ph2}. Thus, by Step 1, for each $k>0$ there exists a renormalized solution $u_k\in E(\mu_k)$ to \eqref{zzz} with datum $\mu=\mu_k$.
		Also, by (5.1) in  \cite{HPanna} we have 
		$${\bf I}_1^{2{\rm diam}(\Om)}(|\mu_k|)\leq \rho_k * {\bf I}_1^{2{\rm diam}(\Om)}(|\mu|)\leq M\, {\bf I}_1^{2{\rm diam}(\Om)}(|\mu|).$$
		Thus by the stability result for  renormalized solutions of \cite{11DMOP}, we can find a subsequence of $\{u_k\}_k$ that converges to  
		 a renormalized solution $u$ to the equation \eqref{zzz} such that 
		 $$|\nabla u(x)|\leq 2C_0 M^{\frac{1}{p-1}}\mathbf{P}[\mu](x) {\rm ~ a.e.}$$
		This completes the proof of the theorem.
		

\begin{remark} Finally, we remark that due to a comparison estimate obtained earlier in \cite{HPanna} for the case $\frac{3n-2}{2n-1}<p\leq 2-1/n$, all results in this paper also hold in this case. Moreover, local and global estimates `below the duality exponent' in the spirit of \cite{M10, Ph3, Ph4} for $1<p\leq 2-1/n$ can also be deduced from the comparison estimates of  Theorems \ref{maininterior} and \ref{mainboundary} and \cite{HPanna}.	 
\end{remark}
	\textbf{Acknowledgments:} Q.H.N.  is supported by the Academy of Mathematics and Systems Science, Chinese Academy of Sciences startup fund, and the National Natural Science Foundation of China (12050410257).	N.C.P. is supported in part by Simons Foundation (award number 426071). The authors would like to thank  Hongjie Dong for helpful comments. 

\end{document}